 \documentclass[12pt]{article}
\usepackage{authblk}
\usepackage{amsfonts,amsmath,amssymb,amsthm,mathrsfs,enumerate}
\usepackage{amsmath,bm}
\usepackage{graphicx}
\usepackage{color,hyperref}
\usepackage{ifpdf}
\usepackage{epstopdf}
\usepackage{subfigure}
\usepackage[numbers,sort&compress]{natbib}
\usepackage{setspace}
\usepackage{indentfirst}
\usepackage{tipa}
\usepackage{esvect} 

\usepackage{booktabs}
\usepackage{tabularx}

\usepackage{enumitem}

\hypersetup{hidelinks}

\usepackage{mathtools} 

\usepackage{graphicx,color, booktabs,multirow}
\usepackage{tikz,chemfig}
\usepackage{graphicx,booktabs,multirow,mathrsfs}

\usepackage{float} 

\usepackage{calc}
\usetikzlibrary{backgrounds}
\usetikzlibrary{shadings,intersections}
\usetikzlibrary{arrows}
\usetikzlibrary{shapes,shapes.geometric,shapes.misc}
\pgfdeclarelayer{edgelayer}
\pgfdeclarelayer{nodelayer}
\usetikzlibrary{automata}
\pgfsetlayers{background,edgelayer,nodelayer,main}
\tikzstyle{none}=[inner sep=0mm]
\tikzstyle{every loop}=[]

\tikzstyle{dotted}=[dash pattern=on \pgflinewidth off 2pt]
\tikzstyle{dashed}=[dash pattern=on 3pt off 3pt]

\usetikzlibrary{backgrounds}
\usetikzlibrary{arrows}
\usetikzlibrary{shapes,shapes.geometric,shapes.misc}
\pgfdeclarelayer{edgelayer}
\pgfdeclarelayer{nodelayer}
\pgfsetlayers{background,edgelayer,nodelayer,main}
\tikzstyle{none}=[inner sep=0mm]
\tikzstyle{every loop}=[]

\tikzstyle{dotted}=[dash pattern=on \pgflinewidth off 2pt]
\tikzstyle{dashed}=[dash pattern=on 3pt off 3pt]

\newcommand \tikzp[2]
{
	\begin{center}
		\begin{tikzpicture}[scale=#1]
			#2
		\end{tikzpicture}
	\end{center}
}

\tikzstyle{new style 0}=[fill=black, draw=black, shape=circle]
\tikzstyle{red style 1}=[fill=red, draw=black, shape=circle]
\tikzstyle{blue style 2}=[fill=blue, draw=black, shape=circle]
\tikzstyle{white style 4}=[fill=white, draw=black, shape=circle]
\tikzstyle{bklack style 5}=[fill=black, draw=black, shape=rectangle]
\tikzstyle{red style 3}=[fill=red, draw=black, shape=rectangle]
\tikzstyle{yellow style 7}=[fill=yellow, draw=black, shape=rectangle]
\tikzstyle{new style 8}=[fill={rgb,255: red,0; green,132; blue,0}, draw={rgb,255: red,0; green,131; blue,0}, shape=circle]

\tikzstyle{new edge style 0}=[-]
\tikzstyle{new edge style 1}=[-, draw=red]
\tikzstyle{new edge style 2}=[-, draw=blue]
\tikzstyle{new edge style 3}=[-, draw={rgb,255: red,0; green,156; blue,0}]

\tikzstyle{cblue}=[circle, draw, thin,fill=blue!20, scale=0.5]

\tikzstyle{new style 0}=[fill=black, draw=black, shape=circle]
\tikzstyle{red style 1}=[fill=red, draw=black, shape=circle]
\tikzstyle{blue style 2}=[fill=blue, draw=black, shape=circle]
\tikzstyle{white style 4}=[fill=white, draw=black, shape=circle]
\tikzstyle{bklack style 5}=[fill=black, draw=black, shape=rectangle]
\tikzstyle{red style 3}=[fill=red, draw=black, shape=rectangle]
\tikzstyle{yellow style 7}=[fill=yellow, draw=black, shape=rectangle]
\tikzstyle{new style 8}=[fill={rgb,255: red,0; green,132; blue,0}, draw={rgb,255: red,0; green,131; blue,0}, shape=circle]

\tikzstyle{new edge style 0}=[-]
\tikzstyle{new edge style 1}=[-, draw=red]
\tikzstyle{new edge style 2}=[-, draw=blue]
\tikzstyle{new edge style 3}=[-, draw={rgb,255: red,0; green,156; blue,0}]

\theoremstyle{definition}

\theoremstyle{plain}
\newtheorem{theorem}{\bf Theorem}[section]
\newtheorem{proposition}[theorem]{\bf Proposition}

\newtheorem{lemma}[theorem]{\bf Lemma}
\newtheorem{corollary}[theorem]{\bf Corollary}
\newtheorem{conjecture}[theorem]{\bf Conjecture}
\newtheorem{problem}[theorem]{\bf Problem}

\newcommand\red[1] {{\color{red} #1}}
\newcommand\green[1] {{\color{green} #1}}

\newcommand\blue[1] {{\bf \color{blue} #1}}

\usepackage{stmaryrd}

\def \vect {\vec{T}}

\newcommand \brk[1]
{
	\llbracket #1\rrbracket
}

\def \R {{\mathbb R_+}}
\def \N {{\mathbb N}}
\def \vecx{\bar {\bf x}}

\def \sett {\vec{\cal T}}

\def \usecolour
{\newcommand {\rered}{\red}
	\newcommand {\reblue} {\blue}
	\newcommand {\regreen} {\green}
}

\usecolour

\newcommand \them[2]
{\begin{theorem}
		\label{#1} #2
	\end{theorem}
}

\newcommand \lemm[2]
{\begin{lemma}
		\label{#1} #2
	\end{lemma}
}

\newcommand \prop[2]
{\begin{proposition}
		\label{#1} #2
	\end{proposition}
}

\newcommand \prom[2]
{\begin{problem}
		\label{#1} #2
	\end{problem}
}

\newcommand \cor[2]
{\begin{corollary}
		\label{#1} #2
	\end{corollary}
}

\newcommand \equ[2]
{\begin{equation}\label{#1}
		#2
	\end{equation}
}

\newcommand \eqn[2]
{\begin{eqnarray}\label{#1}
		#2
	\end{eqnarray}
}

\newcommand \Floor[1]
{
	\left	\lfloor {#1}\right \rfloor 
}

\newcounter{countcase}
\def\incase{\addtocounter{countcase}{1}{\noindent {\bf Case \thecountcase}: }}

\newcounter{countclaim}
\def\inclaim{\addtocounter{countclaim}{1}
	{\vspace{0.2 cm}\noindent {\bf Claim \thecountclaim}: }}

\newcommand{\proofend}{{\hfill$\Box$}}

\newcommand {\relabel}[1] {\label{#1} \red{[*: #1]}}
\newcommand {\rebibitem}[1] {\bibitem{#1} \red{[*: #1]}}

\def\relabel {\label}

\def \scre{{\mathscr W}}

\def \scrl{{\mathscr L}}

\renewcommand {\rebibitem}[1] {\bibitem{#1}}

\begin{document}
	
	\title{A Recursive Characterization of Laplacian Spectral Radii of Trees with Bounded Maximum Degree		
}

\author[1\thanks{Corresponding author. Email: fengming.dong@nie.edu.sg and donggraph@163.com.}]
{Fengming Dong}

\author[2\thanks{Email: ruixuezhang7@163.com.}]
{Ruixue Zhang}

	\affil[1]{\footnotesize National Institute of Education, Nanyang Technological University, Singapore}
	
	\affil[2]{\footnotesize School of Mathematics and Statistics, Qingdao University, China}

\date{}

\maketitle{}

\baselineskip 0.6 cm

\begin{abstract}
	For a positive integer $r$ and a real number $\alpha\ge 2$, let
	$\mathscr L_r(\alpha)$ be the set of positive real numbers
	containing
	$\alpha-1$ and closed under the following operation: if
	$q_1,\ldots,q_s\in\mathscr L_r(\alpha)$, where $1\leq s\leq r-1$, and
	$
	q=\alpha-1-s-\sum\limits_{i=1}^{s}q_i^{-1}>0,
	$
	then $q\in\mathscr L_r(\alpha)$.  We prove that there exists a tree $T$ with
	$\Delta(T)\leq r$ and Laplacian spectral radius $\mu(T)=\alpha$ if and only if
	$(\alpha-1)^{-1}\in\mathscr L_r(\alpha)$.  Consequently, the Laplacian spectral
	radii of nontrivial trees are precisely the real numbers $\alpha\geq2$ satisfying
	$(\alpha-1)^{-1}\in\mathscr L_{\lfloor\alpha\rfloor-1}(\alpha)$.  As an application,
	we prove that, for integers $k\geq2$ and $r\geq2$, a tree $T$ with
	$\mu(T)=k^2$ and $\Delta(T)=r$ exists if and only if
	$(k-1)^2+2\leq r\leq k^2-1$.
\end{abstract}

\vskip 2mm


\vskip 2mm



	
\section{Introduction}

\subsection{Background}

Unless stated otherwise, all graphs considered in this article are simple and undirected.
For a graph $G$, 
let $V(G), E(G)$ and $\Delta(G)$ denote its vertex set, edge set and maximum degree, respectively.
For any non-empty subset $S$ of $V(G)$,
let $G[S]$ denote the subgraph of $G$ induced by $S$.
For any $u\in V(G)$, 
let $N_G(u)$ (or simply $N(u)$) be the set of neighbors of $u$ in $G$,
and let  $d_G(u)$ (or simply $d(u)$) be
the size of $N(u)$, 
called the degree of $u$ in $G$.
For any integer $k>0$, let $\brk{k}$ denote the set $\{1, 2, \ldots, k\}$.

For a simple graph $G=(V, E)$ with $V=\{v_i: i\in \brk{n}\}$, the adjacency matrix $A(G)$ of $G$ according to the vertex ordering $v_1, v_2, \ldots, v_n$ is defined to be the $(0,1)$ matrix $(a_{i, j})_{n\times n}$, where $a_{i, j}=1$ if and only if $v_iv_j\in E$.
The {\it characteristic polynomial}
of $G$ is defined to be 
$p(G,x)=
\det(x I_n-A(G))$,
where $I_n$ is the identity matrix of size $n$. 
It is known that
$p(G,x)$ is independent of
the vertex ordering $v_1,v_2,\dots,v_n$, and 
all roots of $p(G,x)=0$ are real
numbers
(see~\cite{Brouwer, spectra:2010}).
The largest root of $p(G,x)$, 
denoted by $\lambda(G)$,
is called the 
{\it spectral radius} of $G$.

Let $\N$ and $\R$ 
denote the set of positive integers and the set of positive 
real numbers, respectively.
For any  $r\in \N$ and 
 $\alpha\ge 1$, 
 $\scre_r(\alpha)$  
 is the set of nonnegative 
 real numbers 
 containing $\alpha$ and 
 closed under the following operation:
 for any $q_1,q_2,\ldots,q_s\in 
  \scre_r(\alpha)\setminus \{0\}$
  with $q=\alpha-
  \sum\limits_{1\le i\le s}q_i^{-1}$, 
  if either $s<r$ and $q\ge 0$, 
  or $s=r$ and $q=0$, 
  then $q\in \scre_r(\alpha)$.
In \cite{dong2024}, 
we showed that 
there exists a tree $T$ with maximum degree 
$\Delta(T)\le r$
and 
$\lambda(T)=\alpha$
if and only if $\alpha^{-1}\in \scre_r(\alpha)$.
It follows that the set of 
spectral radii of non-trivial trees 
is exactly the set of  $\alpha\in \R$ such that 
$0\in \scre_r(\alpha)$
for $r=\Floor{\alpha^2}$.
Applying this conclusion, we  proved that for any positive integers $r$ and $k$, where $r\ge 2$, 
there exists a tree $T$ 
with $\Delta(T)=r$ and $\lambda(T)=\sqrt k$ 
if and only if $\frac 14 k+1<r\le k$.

The \textit{Laplacian matrix} of $G$,
denoted by $L(G)$, 
is defined to be 
$D(G)-A(G)$, 
where 
$D(G)=\operatorname{diag}(d(v_1),\ldots,d(v_n))$.
The {\it Laplacian spectral radius} of $G$, 
denoted by $\mu(G)$, 
is the largest eigenvalue of $L(G)$. 

In this article, we focus on 
the Laplacian spectral radii of 
trees.
Note that $K_1$ and $K_2$ are the only trees of orders $1$ and $2$, respectively, and their Laplacian spectral radii are 
$0$ and $2$, respectively. Now let $T$ be a tree of order $n\geq 3$. The fundamental order-extremal results state that
\equ{1}
{
2+2\cos\left(\frac{\pi}{n}\right)
=\mu(P_n)
\leq \mu(T)
\leq \mu(K_{1,n-1})
=n,
}
where equality on the left holds if and only if $T\cong P_n$, and equality on the right holds if and only if $T\cong K_{1,n-1}$;
see  ~\cite{PetrovicGutman2002,Gutman2002}. In particular, since $n\geq 3$,
$$
\mu(T)\geq 2+2\cos\left(\frac{\pi}{n}\right)
\geq 3>2.
$$

In terms of the maximum degree, the general lower bound of Grone and Merris
\cite{GroneMerris1994}, together with its equality characterization
\cite{ZhangLuo2002}, and the tree-specific upper bound of Stevanovi\'c \cite{Stevan2003}
yield
\begin{equation}
	\Delta(T)+1
	\leq \mu(T)
	< \Delta(T)+2\sqrt{\Delta(T)-1}
	\label{eq-1}
\end{equation}
for every tree with $\Delta(T)\geq 2$.
Moreover, equality in the lower bound holds precisely when $T$ is a star~\cite{ZhangLuo2002}.


A substantial body of work determines or characterizes the trees of a fixed order
having maximum Laplacian spectral radius subject to an additional prescribed
parameter, such as 
the number of pendant vertices~\cite{HongZhang2005}, matching number~\cite{Guo2003},  diameter~\cite{Guo2006Diameter}, maximum degree~\cite{YuLu2008}, and degree sequence~\cite{Zhang2008DegreeSequence}.
Related extremal and ordering results have also been obtained for trees with perfect matchings~\cite{YuanShaoHe2009,ChenQian2013}.  

\subsection{Problems and results}

\def \lsrt {\mathcal{R}_{L}^{\mathrm{nt}}}

Let $\lsrt
:=\{\mu(T):T\text{ is a non-trivial tree}\}$.
To the best of our knowledge, relatively little is known about the set $\lsrt$. It is therefore natural to ask whether this set admits a characterization, particularly from an algebraic perspective.

In this article, we 
provide a recursive characterization of 
the elements of $\lsrt$.

For any $r\in \N$ and  
$\alpha\ge 2$,  define 
the set $\scrl^{(n)}_r(\alpha)$ 
for each integer $n\ge 0$ as follows:
 \begin{enumerate}
 	[itemsep=0mm, label=(\roman*)]
 	\item $\scrl^{(0)}_r(\alpha)=\{\alpha-1\}$; 
 	
 	\item for any $n\in \N$, 
 	$\scrl^{(n)}_r(\alpha)$ is the 
 	smallest set with the following properties
 	that 
 	$\scrl^{(n-1)}_r(\alpha)
 	\subseteq \scrl^{(n)}_r(\alpha)$,
 	and 
 	for any 
 	$q_1,q_2,\ldots,q_s\in \scrl^{(n-1)}_r(\alpha)$, 
 	where $1\le s\le r-1$, 
 	if $q=\alpha-1-s
 	-\sum\limits_{1\leq i\leq s}q_i^{-1}>0$, then $q\in 
 	\scrl^{(n)}_r(\alpha)$.
 \end{enumerate}
By definition, every element
of  $\scrl^{(n)}_r(\alpha)$
is positive.
 Obviously, 
 $\scrl^{(n-1)}_r(\alpha)
 \subseteq \scrl^{(n)}_r(\alpha)$
 for every $n\in \N$
 and $\alpha\ge 2$.
 Now let 
 	\equ{new-al}
 {
 	\scrl_r(\alpha)
 	=\bigcup_{n=0}^{\infty}
 	\scrl^{(n)}_r(\alpha).
}

It is not difficult to verify that 
$\scrl_1(\alpha)=\{\alpha-1\}$ for all $\alpha\ge 2$,
$\scrl_{2}(2)=\{1\}$, $\scrl_{2}(3)=\{2, \frac{1}{2}\}$ and $\scrl_2(4)
=\left \{\frac{2i+1}{2i-1}: 
i\in \N\right \}$.
By definition, 
$\scrl_{r_1}(\alpha)\subseteq 
\scrl_{r_2}(\alpha)$ holds
for any $r_1,r_2\in \N$ with $r_1<r_2$ and $\alpha\ge 2$.

Using $\scrl_r(\alpha)$,
 we will establish 
the following two results 
on the existence of a tree
with prescribed maximum degree and Laplacian 
spectral radius.
One of them provides an equivalent condition for the existence of a tree $T$ with $\Delta(T)=r$ and $\mu(T)=\alpha$, 
and the other one 
 gives an equivalent condition for the existence of a tree $T$ with $\Delta(T)\le r$ and $\mu(T)=\alpha$.

\them{se4-th}
{
	For any $r\in \N$
	and  
	$\alpha\ge 2$, 
	 there exists a tree $T$ 
	with $\Delta(T)=r$ and $\mu(T)=\alpha$
	if and only if 
	there is a multiset $\{q_1,q_2,\ldots,q_r\}$ 
	of elements of $\scrl_r(\alpha)$ 
	such that 
	$\sum\limits_{i=1}^r q_i^{-1}=\alpha-r$.
}

\begin{theorem}\label{thm-1}
	For any $r\in \N$ and
	 $\alpha\ge 2$,
	there exists a tree $T$ with $\Delta(T)\leq r$ and $\mu(T)=\alpha$ if and only if $(\alpha-1)^{-1}\in \scrl_r(\alpha)$.
\end{theorem}

If $\alpha=2$, then 
$\alpha\in \lsrt$ and $1=(\alpha-1)^{-1}\in \scrl_1(\alpha)$.
For $\alpha>2$, 
if $T$ is a tree with $\mu(T)=\alpha$, then 
by  (\ref{eq-1}), 
$\Delta(T)\le \alpha-1$,
implying that 
$(\alpha-1)^{-1}\in \scrl_{\Floor{\alpha}-1}(\alpha)$
by
Theorem~\ref{thm-1}.
Conversely, 
if $(\alpha-1)^{-1}\in \scrl_{\lfloor\alpha\rfloor-1}(\alpha)$,
then Theorem~\ref{thm-1} supplies a nontrivial tree $T$ with $\mu(T)=\alpha$.
Hence the next conclusion follows.

\cor{se1-co1}
{
The set $\lsrt$ 
is exactly
the set of  $\alpha\ge 2$ such that 
$(\alpha-1)^{-1}\in \scrl_r(\alpha)$,
where $r=\Floor{\alpha}-1$.
}

By modifying the definition of 
$\scrl_r(\alpha)$, we 
define the set $\scrl(\alpha)$ 
recursively:  
$\alpha-1\in \scrl(\alpha)$; and for any multiset $\{q_1, q_2, \ldots, q_s\}$ of 
elements of 
$\scrl(\alpha)$,
where $s\in \N$, 
	if $q:=\alpha-1-s-\sum\limits_{1\leq i\leq s}q_i^{-1}>0$, 
	then $q\in \scrl(\alpha)$.
	It is not difficult to verify that 
	$$
	\scrl(\alpha)
	=\bigcup_{r\ge 1}\scrl_r(\alpha).
	$$


\cor{se1-co2}
{
	The set $\lsrt$ 
	is exactly
	the set of  $\alpha\ge 2$ such that 
	$(\alpha-1)^{-1}\in \scrl(\alpha)$.
}

\proof For any $\alpha\in \lsrt$, 
by Corollary~\ref{se1-co1},
$(\alpha-1)^{-1}\in \scrl_r(\alpha)$ for $r=\lfloor \alpha\rfloor -1$, and 
thus $(\alpha-1)^{-1}\in \scrl(\alpha)$.
Conversely, 
if $(\alpha-1)^{-1}\in
\scrl(\alpha)$, 
then 
$(\alpha-1)^{-1}\in
\scrl_{r}(\alpha)$
for some $r\in \N$, and 
Theorem~1.2 gives a tree $T$ with
$\Delta(T)\le r$ and 
$\mu(T)=\alpha$, 
implying that $\alpha\in \lsrt$.
Thus, Corollary~\ref{se1-co2} 
holds.
\proofend 

Let $T$ be a tree with 
$\mu(T)=\alpha$ and $\Delta(T)=r\ge 2$.
 Then 
(\ref{eq-1}) can be rewritten as 
 the inequality $\alpha+2-2\sqrt{\alpha}<r\leq \alpha-1$. 
This inequality is a necessary condition for the existence of a tree $T$ such that $\mu(T)=\alpha$ and $\Delta(T)=r$.
It is natural to ask 
whether 
for any $\alpha\in \lsrt$, 
this condition 
is sufficient 
for the existence of a tree $T$ 
with $\mu(T)=\alpha$ and $\Delta(T)=r$.


Applying Theorem~\ref{se4-th}, 
we show that in the case $\alpha=k^2$, where $k\in \N$,
this condition is sufficient 
for the existence of a tree $T$ 
with $\mu(T)=k^2$ and $\Delta(T)=r$.

\begin{theorem}\label{thm-2}
Let $\alpha=k^2$, where 
$k\in \N$ with $k\ge 2$.
Then,  for any $r\in \N$,
there exists a tree $T$ such that 
$\Delta(T)=r$ and $\mu(T)=k^2$
if and only if 
$(k-1)^2+2\le r\le k^2-1$.
\end{theorem} 

The remainder of the article is organized as follows. In Section~\ref{sec2}, we establish the necessity of Theorems~\ref{se4-th} and~\ref{thm-1}. In Section~\ref{sec3}, we prove the remaining parts, thereby completing the proofs of both theorems. We then apply Theorem~\ref{se4-th} in Section~\ref{sec4} to prove Theorem~\ref{thm-2}. Finally, Section~\ref{sec5} presents a problem for further study.

\section{
 The necessity 
of Theorems~\ref{se4-th}
and~\ref{thm-1}
	\label{sec2}
}

\subsection{A condition 
	guaranteeing $\mu(G)=\alpha$}

Let $G$ be a simple graph and 
let $Q(G)=D(G)+A(G)$. 
Merris ~\cite{Merr-1994} showed that the largest eigenvalue $\nu(G)$ of
$Q(G)$ satisfies $\nu(G) = 
\mu(G)$ whenever $G$ is bipartite.

\begin{theorem}
	[\cite{Merr-1994}]
	\label{Merr-1994}
If $G$ is a bipartite graph, then 
$\nu(G)=\mu(G)$. 		
\end{theorem}

If $G$ is connected, then $Q(G)$ 
is irreducible and  the following 
conclusion follows from the
 Perron-Frobenius theorem
 \cite{Frob:1912,Perron:1907}.
 
 \begin{theorem}
 	[The Perron-Frobenius theorem]\label{thm-PF}
 	Let $G$ be a connected graph of order at least $2$. 
 	The largest eigenvalue $\nu(G)$ 
 	of $Q(G)$ 
 	has multiplicity 1 and 
there is an eigenvector 
$\vecx$ 
of $Q(G)$ corresponding to $\nu(G)$ in which all 
 	components are positive. 
 	Furthermore, $\nu(G)$ is the only eigenvalue of $Q(G)$ with this property.
 \end{theorem}

Applying Theorems \ref{Merr-1994} and  \ref{thm-PF}, we obtain a necessary and sufficient 
condition for $\mu(G)=\alpha$,
where $G$ is a connected bipartite graph.

\begin{proposition}\label{prop-1}
Let $G$ be a connected  bipartite graph of order at least $2$ 
and $\alpha\in \R$.
Then, $\mu(G)=\alpha$ if and only if there is a mapping $\phi$ from $V(G)$ to $\mathbb{R}^+$ such that for each $v\in V(G)$,
\begin{equation}\label{eq-3}
(\alpha-d_G(v))\phi(v)=\sum_{w\in N_G(v)}\phi(w).
\end{equation}
\end{proposition}

\proof Assume that 
$V(G)=\{v_i: i\in \brk{n}\}$, where $n\ge 2$,
and $Q(G)$ is a 
matrix 
obtained according to the vertex ordering $v_1, v_2,\dots, v_n$ of $G$. 

Assume that 
$\mu(G)=\alpha$. 
Since $G$ is bipartite, 
$\nu(G)=\alpha$ by Theorem~\ref{Merr-1994}. 
Then, 
by Theorem~\ref{thm-PF},
there exists a mapping $\phi: V(G)\rightarrow \R$ such that 
$Q(G)\vecx =\alpha \vecx$,
where $\vecx=(\phi(v_1), \dots, \phi(v_n))^{\intercal}$.
It follows that for each $i\in \brk{n}$,
\equ{eq2}
{
d(v_i)\phi(v_i)
+\sum_{v_j\in N(v_i)}\phi(v_j)
=\alpha \phi(v_i).
}
Thus, (\ref{eq-3}) holds for each vertex $v$ in $G$.

On the other hand, if there is a mapping 
$\phi: V(G)\rightarrow \R$ such that (\ref{eq-3}) holds for each vertex $v$ in $G$, 
then (\ref{eq2}) holds for 
each $i\in \brk{n}$, 
implying that $Q(G)\vecx =\alpha \vecx$,
where $\vecx=(\phi(v_1), \dots, \phi(v_n))^{\intercal}$.
By Theorem~\ref{thm-PF}, 
$\alpha=\nu(G)$. 
Then, Theorem~\ref{Merr-1994}
implies that $\mu(G)=\alpha$. 
\proofend 

\subsection{The necessary parts of 
Theorems~\ref{se4-th}
and~\ref{thm-1}
}

Since every tree
is  bipartite, 
Proposition~\ref{prop-1}
holds for $T$.
We will apply Proposition~\ref{prop-1}
to show that 
$(\alpha-1)^{-1}\in \scrl_r(\alpha)$,
where $\alpha=\mu(T)$ and 
$r=\Delta(T)$.

We first introduce a set 
$\sett$
of directed trees. 
For a directed tree $\vect$
and $v\in V(\vect)$,
let $N^-_{\vect}(v)$ denote the set of  in-neighbors of $v$ in $\vect$
(i.e., the set of vertices 
$w$ so that 
$(w,v)$ is an arc in $\vect$), 
and 
let $id_{\vect}(v)$ denote 
the in-degree 
of $v$ in $\vect$
(i.e., $id_{\vect}(v)=|N^-_{\vect}(v)|$).
If $id_{\vect}(v)=0$, then $v$ is called 
a source of $\vect$.
Let $\sett$ be the set of directed trees $\vect$ 
that have exactly one source.
Clearly, for any $\vect\in \sett$, 
if $u$ is the unique source of $\vect$, 
then $id_{\vect}(v)=1$ 
for each $v\in V(\vect)\setminus \{u\}$.
Note that 
for any undirected tree 
$T$ and $u\in V(T)$, 
 there exists a unique orientation $\vect$ of $T$ 
 that has $u$
as its unique source.

The recursive definition of $\scrl_r(\alpha)$ has a natural interpretation in
terms of ratios of entries of a positive eigenvector.  Let $T$ be a tree
rooted at $u$, orient every edge away from $u$, and suppose that
$\varphi:V(T)\to\mathbb \R$ satisfies condition (\ref{eq-3}).  For every non-source
vertex $v$, let $v'$ be its parent and define
\[
\Psi(v):=\frac{\varphi(v')}{\varphi(v)}.
\]
Proposition~\ref{pro2.6} shows that each such ratio belongs to $\scrl_r(\alpha)$; thus the
set $\scrl_r(\alpha)$ captures the ratios that arise along oriented edges.

\prop{pro2.6}
{
	Let $T$ be a  tree  with $|V(T)|\ge 2$, 
	$\alpha=\mu(T)$
	and $r=\Delta(T)$.
	Assume that  
	$\phi: V(T)\rightarrow \R$
	is a mapping 
	such that for each $v\in V(T)$, 
	\begin{equation}\label{pro2.6-e1}
		(\alpha-d_T(v))\phi(v)=\sum_{w\in N_T(v)}\phi(w).
	\end{equation}
	Then the following conclusions hold:
	\begin{enumerate}
		[itemsep=0mm, label=(\roman*)]
		\item 
		$\frac{\phi(v')}{\phi(v)}\in \scrl_r(\alpha)$
		for every ordered pair 
		$(v', v)$ of adjacent vertices in $T$; 
		\item $(\alpha-1)^{-1}\in \scrl_r(\alpha)$; 
		and 
		\item if $s=d_T(u)$ for some vertex $u$ in $T$,  
		then there exists 
		a multiset $\{q_1,q_2,\ldots,q_s\}$ 
		of elements of $\scrl_r(\alpha)$ 
		such that $\sum\limits_{i=1}^s q_i^{-1}
		=\alpha-s$.
	\end{enumerate}
}

\proof 
We prove the following claims.
Let $\vect$ be any orientation of $T$ with a unique source,
and thus $\vect\in \sett$.
Clearly, $id_{\vect}(v)=1$ for 
each vertex in $\vect$, 
unless $v$ is the only source 
of $\vect$.
For each $v$ with $id_{\vect}(v)=1$, let 
$\Psi(v):=\frac{\phi(v')}{\phi(v)}$,
where $v'$ is the only vertex in $\vect$ such that $(v',v)$ 
is a directed edge in $\vect$.

For any $v\in V(\vect)$,
let $N^+_{\vect}(v)$ denote the set of out-neighbors of $v$ in $\vect$
(i.e., the set of vertices $w$ so that 
$(v,w)$ is an arc in $\vect$), 
and 
let $od_{\vect}(v)$ denote the out-degree of $v$ in $\vect$ 
(i.e., $od_{\vect}(v)=|N^+_{\vect}(v)|$ ).

\inclaim $\Psi(v)=\alpha-1$ 
for each $v\in V(\vect)$
with $od_{\vect}(v)=0$.

Since $od_{\vect}(v)=0$, 
$v'$ is the unique neighbor of $v$ in $T$. By (\ref{pro2.6-e1}),
\equ{pro2.6-e2}
{
	(\alpha-1)\phi(v)=\phi(v'),
}
implying that $\Psi(v)=\frac{\phi(v')}{\phi(v)}
=\alpha-1$.
Claim 1 holds.

\inclaim $\Psi(v)\in \scrl_r(\alpha)$ for each 
$v\in V(T)$ with $id_{\vect}(v)=1$.

We prove this claim
upward from the leaves. 
First, 
by Claim 1, $\Psi(v)=\alpha-1\in \scrl_r(\alpha)$  if $od_{\vect}(v)=0$.
Now assume that $v$ is a vertex 
in $T$  with 
$id_{\vect}(v)=1$ and 
$N^+_{\vect}(v)
=\{v_1,\ldots,v_t\}$,
where $t:=od_{\vect}(v)>0$, 
such that 
$\Psi(v_i)\in \scrl_r(\alpha)$ for each $i\in \brk{t}$.
By the assumption, 
$\Psi(v_i)=\frac{\phi(v)}{\phi(v_i)}$.
By (\ref{pro2.6-e1}), 
\begin{equation}
	\label{pro2.6-e3}
	(\alpha-t-1)\phi(v)
	=\phi(v')+
	\sum_{i=1}^t \phi(v_i).
\end{equation}
It follows that 
\equ{pro2.6-e4}
{
	\Psi(v)=\frac{\phi(v')}{\phi(v)}
	=\alpha-1-t-\sum_{i=1}^t
	\frac{\phi(v_i)}{\phi(v)}
	=\alpha-1-t-\sum_{i=1}^t
	\Psi(v_i)^{-1}.
}
As 
$t=d_T(v)-1\leq r-1$, 
by the definition of $\scrl_r(\alpha)$, $\Psi(v)\in \scrl_r(\alpha)$.
Hence Claim 2 holds.

\inclaim Statements (i) and (ii) hold.

For every ordered pair 
$(v', v)$ of adjacent vertices in $T$,  root the tree at $v'$ 
to get $\vect\in \sett$
so that 
$(v',v)$
is a directed edge in $\vect$.

By Claim 2, $
\frac{\phi(v')}{\phi(v)}=\Psi(v)\in \scrl_r(\alpha)$.
Thus, statement (i) holds.


By Claim 1, 
if $v_1v_2\in E(T)$ and $v_1$ is a leaf of $T$, 
then $\frac{\phi(v_2)}{\phi(v_1)}=
\alpha-1$,
implying that 
$\frac{\phi(v_1)}{\phi(v_2)}=
(\alpha-1)^{-1}$.
Thus, by  the conclusion of (i), 
we have 
$(\alpha-1)^{-1}
\in \scrl_r(\alpha)$.
Hence (ii) holds and 
Claim 3 follows.

Now assume that $u$ is a vertex in $T$ and $s=d_T(u)$.
We may assume that
$\vect$ is the orientation of $T$ 
such that $u$ is its unique source. Let
$N^+_{\vect}(u)
=\{u_1,\ldots,u_s\}$,
where $N^+_{\vect}(v)$
denotes the set of out-neighbors of a vertex $v$ in $\vect$.
Thus, by Claim 2, 
$q_i:
=\frac{\phi(u)}{\phi(u_i)}
=\Psi(u_i)$ for each $i\in \brk{s}$.

\inclaim $q_i\in \scrl_r(\alpha)$ 
for each $i\in \brk{s}$ and  $\sum\limits_{i=1}^{s}q_i^{-1}=\alpha-s$.

By Claim 2, $q_i=\Psi(u_i)\in \scrl_r(\alpha)$ for each $i\in \brk{s}$.
By (\ref{pro2.6-e1}), 
\begin{equation}
	\label{pro2.6-e5}
	(\alpha-s)\phi(u)
	=\sum_{i=1}^s \phi(u_i).
\end{equation}
It follows that 
\equ{}
{
	\sum_{i=1}^{s}q_i^{-1}
	=\sum_{i=1}^s \frac{\phi(u_i)}{\phi(u)}
	=\alpha-s.
}
Claim 4 holds.
Thus (iii) follows.
Hence the proposition holds.
\proofend

By Propositions~\ref{prop-1} and~\ref{pro2.6},
the following conclusion holds,
and thus the necessary parts of 
Theorems~\ref{se4-th}
and~\ref{thm-1} follow.

\cor{se2-co1}
{
Let $T$ be a tree of order at least $2$, 
and let $\alpha=\mu(T)$
and $r=\Delta(T)$.
Then $(\alpha-1)^{-1}\in \scrl_r(\alpha)$,
and for $s=d_T(u)$,
where $u$ is any vertex in $T$, 
there exists 
a multiset $\{q_1,q_2,\ldots,q_s\}$ 
of elements of $\scrl_r(\alpha)$ 
such that $\sum\limits_{i=1}^s q_i^{-1}
=\alpha-s$.
}

\section{Proofs of Theorems  \ref{se4-th} and \ref{thm-1}
\label{sec3}
}

\subsection{A mapping 
	$\Phi$ on $V(\vect)$ for $\vect\in \sett$}

Recall that $\sett$ is the set of directed trees $\vect$ with
exactly one source $u$.
For any real number $\alpha\ge 2$, 
let $\sett_{\alpha}$
denote the set of directed trees
$\vect$ with a unique 
source $u$
and a map $\Phi: V(\vect) \rightarrow \{-1\}\cup \R$ that 
satisfies 
\begin{equation}\label{eq-7}
\Phi(v)=
\alpha-1-od_{\vect}(v)
-\sum\limits_{w\in N^+_{\vect}(v)}
\Phi(w)^{-1}
\end{equation}
 for each 
 $v\in V(\vect)$,
 and $\Phi(v)>0$ for 
 every $v\ne u$.
If $v$ is a sink of $\vect$, then  $od_{\vect}(v)=0$
and $N^+_{\vect}(v)=\emptyset$,
implying that 
$\Phi(v)=\alpha-1$.

Note that for any $\alpha\ge 2$,
if $\vect\in \sett_{\alpha}$,
then the map $\Phi$ 
satisfying condition (\ref{eq-7})
is actually
uniquely determined by 
$\vect$ and $\alpha$. 
Let $\Phi_{\alpha, {\vect}}$ denote this map $\Phi$.

The one-vertex directed tree belongs to $\sett_\alpha$ for every $\alpha\ge2$.  We next determine $\sett_2$.  Let $\vect\in\sett_2$ have at least two vertices, let $v$ be a sink, and let $v'$ be its parent.  Since $\Phi_{2,\vect}(v)=1$, equation~(\ref{eq-7}) yields
\[
\Phi_{2,\vect}(v')
=1-od_{\vect}(v')-
\sum_{w\in N^+_{\vect}(v')}\Phi_{2,\vect}(w)^{-1}
\le -1.
\]
The definition of $\sett_2$ therefore forces $v'$ to be the source and
$\Phi_{2,\vect}(v')=-1$.  Equality also forces $od_{\vect}(v')=1$, so $v$ is its only child.  Hence, up to isomorphism, the one-vertex directed tree and the directed tree of order two are the only members of $\sett_2$.

 For the
directed tree of order two shown in Figure~\ref{F6}(a), the source value is
$\alpha-2-(\alpha-1)^{-1}$; hence this tree belongs to $\sett_\alpha$ precisely
when $\alpha=2$ or $\alpha>(3+\sqrt5)/2$.  
If $\vect$ is the orientation of
$K_{1,3}$ whose unique source is a leaf, then $\vect$ belongs to both
$\sett_5$ and $\sett_4$, as shown in Figures~~\ref{F6}(b) and~~\ref{F6}(c), respectively, but
it does not belong to $\sett_2$.
\begin{figure}[h]
		\begin{tikzpicture}
		\tikzset{vertex/.style = {shape=circle,fill, draw,minimum size=1em}}
		\tikzset{edge/.style = {->,> = latex'}}

		\tikzstyle{vertex}=[circle, fill=black, inner sep=0pt, minimum size=8pt]

		\node[vertex, label=above:$\alpha-2-\frac1{\alpha-1}$] (a) at (0,0.8) {};
		\node[vertex, label=below:$\alpha-1$] (b) at (0,-0.5) {};

		\draw[edge, thick] (a) to (b);
	
	\end{tikzpicture}
	\hspace{0.1\textwidth}
	\centering
	\centering
	\begin{tikzpicture}
		\tikzset{vertex/.style = {shape=circle,fill, draw,minimum size=1em}}
		\tikzset{edge/.style = {->,> = latex'}}

		\tikzstyle{vertex}=[circle, fill=black, inner sep=0pt, minimum size=8pt]

	\node[vertex, label=above:$\frac{7}{3}$] (a) at (0,0) {};
		\node[vertex, label=above:$\frac{3}{2}$] (b) at (1.5,0) {};
		\node[vertex, label=above:$4$] (c) at (3,0.8) {};
		\node[vertex, label=below:$4$] (d) at (3,-0.8) {};
		
		 \draw[edge, thick] (a) to (b);
		 \draw[edge, thick] (b) -- (c);
		 \draw[edge, thick] (b) -- (d);
	\end{tikzpicture}
\hspace{0.1\textwidth}
	\centering
	\begin{tikzpicture}

				\tikzset{vertex/.style = {shape=circle,fill, draw,minimum size=1em}}
		\tikzset{edge/.style = {->,> = latex'}}

		\tikzstyle{vertex}=[circle, fill=black, inner sep=0pt, minimum size=8pt]

		\node[vertex, label=above:$-1$] (a) at (0,0) {};
		\node[vertex, label=above:$\frac{1}{3}$] (b) at (1.5,0) {};
		\node[vertex, label=above:$3$] (c) at (3,0.8) {};
		\node[vertex, label=below:$3$] (d) at (3,-0.8) {};
		
		 \draw[edge, thick] (a) to (b);
		 \draw[edge, thick] (b) -- (c);
		 \draw[edge, thick] (b) -- (d);
	\end{tikzpicture}

\centerline{
	(a) A member of $\sett_{\alpha}$
	\hspace{2 cm} 
	(b) A member of $\sett_5$
		\hspace{2 cm} 
	(c) A member of $\sett_4$
	}
	
	\caption{In (a),  $\alpha=2$ or 
		$\alpha>\frac{3+\sqrt 5}2$,
		and in (b) and (c), 
		each directed tree is
		an orientation of $K_{1,3}$.
}

\relabel{F6}
\end{figure}

For any $r\in \N$, 
 $\alpha\ge 2$
and $q\in \{-1\}\cup \R$, 
let $\sett_{r, \alpha}(q)$ denote the set of trees 
$\vect\in \sett_{\alpha}$ such that 
\begin{enumerate}
	[itemsep=0mm, label=(\roman*)]
	\item 
	$\Phi_{\alpha,
		\vect}(u)=q$,
	where $u$ is the unique source of $\vect$;
	
	\item $od_{\vect}(u)\le r$ when $q=-1$, and 
	$od_{\vect}(u)\le r-1$ when $q>0$; 
	and 
	\item  
	$od_{\vect}(v)\le r-1$
	for each 
	$v\in V(\vect)\setminus \{u\}$.
\end{enumerate}

For example, 
the underlying directed
tree displayed in Figures \ref{F6}(b) and \ref{F6}(c)
belongs to both $\sett_{3, 5}(\frac{7}{3})$ and $\sett_{3, 4}(-1)$. 
Also, two directed trees of $\sett_{3, 3+\sqrt{2}}(-1)$ are shown in Figure \ref{fig3}. 

\begin{figure}[htp]
	\centering
	\begin{tikzpicture}
		
		\tikzset{vertex/.style = {shape=circle,fill, draw,minimum size=1em}}
		\tikzset{edge/.style = {->,> = latex'}}

		\tikzstyle{vertex}=[circle, fill=black, inner sep=0pt, minimum size=8pt]


	\node[vertex, label=left:$-1$] (a) at (0,0) {};
	\node[vertex, label=above:$2\sqrt{2}-2$] (b) at (2,1) {};
	\node[vertex, label=below:$2\sqrt{2}-2$] (c) at (2,-1) {};
	\node[vertex, label=right:$2+\sqrt{2}$] (d) at (4,1.5) {};
	\node[vertex, label=right:$2+\sqrt{2}$] (e) at (4,0.5) {};
	\node[vertex, label=right:$2+\sqrt{2}$] (f) at (4,-0.5) {};
	\node[vertex, label=right:$2+\sqrt{2}$] (g) at (4,-1.5) {};
		
		 \draw[edge, thick] (a) to (b);
		 \draw[edge, thick] (b) -- (d);
		 \draw[edge, thick] (b) -- (e);
		 \draw[edge, thick] (a) -- (c);
		 \draw[edge, thick] (c) -- (f);
		 \draw[edge, thick] (c) -- (g);
	\end{tikzpicture}
\hspace{0.1\textwidth}
	\centering
	\begin{tikzpicture}

				\tikzset{vertex/.style = {shape=circle,fill, draw,minimum size=1em}}
		\tikzset{edge/.style = {->,> = latex'}}

		\tikzstyle{vertex}=[circle, fill=black, inner sep=0pt, minimum size=8pt]

		\node[vertex, label=left:$-1$] (a) at (0,0) {};
		\node[vertex, label=above:$\frac{3}{2}\sqrt{2}$] (b) at (2,1.5) {};
		\node[vertex, label=above:$\frac{3}{2}\sqrt{2}$] (c) at (2,0) {};
		\node[vertex, label=above:$\frac{3}{2}\sqrt{2}$] (d) at (2,-1.5) {};
		\node[vertex, label=right:$2+\sqrt{2}$] (e) at (4,1.5) {};
		\node[vertex, label=right:$2+\sqrt{2}$] (f) at (4,0) {};
		\node[vertex, label=right:$2+\sqrt{2}$] (g) at (4,-1.5) {};
		
		 \draw[edge, thick] (a) to (b);
		 \draw[edge, thick] (a) -- (c);
		 \draw[edge, thick] (a) -- (d);
		 \draw[edge, thick] (b) -- (e);
		 \draw[edge, thick] (c) -- (f);
		 \draw[edge, thick] (d) -- (g);
	\end{tikzpicture}

\centerline{(a) ~~~~~~~~~~~~ 
	\hspace{5 cm}~~~~ (b) 
	  ~~~~}
	
	\caption{Two directed trees in the set $\sett_{3, 3+\sqrt{2}}(-1)$
}

\label{fig3}
\end{figure}

\subsection{Directed trees $\vect$
in $\sett_{r,\alpha}(-1)$}

\begin{lemma}\label{lem3.1}
Let $r\in \N$ and 
 $\alpha\ge 2$.
For any $q\in \{-1\}\cup \mathbb{R}^+$,
if $\vect\in \sett_{r, \alpha}(q)$,
then 
$\Phi_{\alpha, \vect}(v)\in \scrl_r(\alpha)$ 
for each $v\in V(\vect)$, 
unless 
 $v$ is the source of $\vect$ 
and $q=-1$.
\end{lemma}

\begin{proof}
	We
	proceed by induction upward from the leaves.
	For any $v\in V(\vect)$ with 
	$od_{\vect}(v)=0$, 
	we have 
	$\Phi_{\alpha, \vect}(v)=\alpha-1\in \scrl_r(\alpha)$.

	Suppose that $v$ is not the source, $1\le od_{\vect}(v)\le r-1$, and the assertion holds for every child of $v$.  Since $\Phi_{\alpha,\vect}(v)>0$ by the definition of $\sett_{\alpha}$, equation~(\ref{eq-7}) and the recursive definition of $\scrl_r(\alpha)$ imply that $\Phi_{\alpha,\vect}(v)
	\in\scrl_r(\alpha)$.  The same argument applies to the source when $q>0$, because then its out-degree is at most $r-1$.

	Hence the conclusion holds.
\end{proof}

\begin{figure}[h]
	\centering
	\begin{tikzpicture}[
		>=Stealth,
		vertex/.style={circle,fill=black,inner sep=2pt},
		branch/.style={
			draw,
			dashed,
			rounded corners,
			minimum width=1.7cm,
			minimum height=1.8cm
		},
		every node/.style={font=\small}
		]
		
		\node[vertex,label=left:{$u$}] (u) at (0,2.7) {};
		
		\node[branch] (T1) at (-3,0) {};
		\node[branch] (T2) at (0,0) {};
		\node[branch] (Tm) at (3,0) {};
		
		\node[vertex,label=above left:{$u_1$}] (u1) at (-3,0.75) {};
		\node[vertex,label=above right:{$u_2$}] (u2) at (0,0.75) {};
		\node[vertex,label=above right:{$u_s$}] (um) at (3,0.75) {};
		
		\node at (-3,-0.25) {$\vect_1$};
		\node at (0,-0.25) {$\vect_2$};
		\node at (3,-0.25) {$\vect_s$};
		
		\node at (1.5,0) {$\cdots$};
		
		\draw[->] (u) -- (u1);
		\draw[->] (u) -- (u2);
		\draw[->] (u) -- (um);
		
	\end{tikzpicture}
	\caption{The branch-joining construction. Each dashed box represents
		a directed tree $\vec T_i$ with unique source $u_i$. A new source $u$
		is added, together with the arcs $(u,u_i)$ for
		$i\in\brk{s}$.}
	\label{branch}
\end{figure}

\begin{lemma}\label{lem-3.0}
	Let $r\in \N$ and 
	 $\alpha\ge 2$.
	Then, $\sett_{r, \alpha}(q)\neq \emptyset$ for each $q\in \scrl_r(\alpha)$.
\end{lemma}

\begin{proof}
We prove by induction on $n$ that $\sett_{r,\alpha}(q)\ne\varnothing$ for every
$q\in\scrl_r^{(n)}(\alpha)$.  If $n=0$, then $q=\alpha-1$, and the one-vertex directed tree belongs to $\sett_{r,\alpha}(q)$.  Assume the assertion holds at level $n$ and let $q\in\scrl_r^{(n+1)}(\alpha)$.  If $q\in\scrl_r^{(n)}(\alpha)$, there is nothing to prove.  Otherwise, there exist $1\le s\le r-1$ and $q_1,\ldots,q_s\in\scrl_r^{(n)}(\alpha)$ such that
\equ{branch-e1}
{
q=\alpha-1-s-\sum_{j=1}^{s}q_j^{-1}>0.
}
For each $j\in \brk{s}$, choose $\vect_j\in\sett_{r,\alpha}(q_j)$ by the induction hypothesis.  Form $\vect$ by adding a new source $u$ and the arcs from $u$ to the sources of $\vect_1,\ldots,\vect_s$,
as shown in Figure~\ref{branch}.
 By (\ref{eq-7}) and (\ref{branch-e1}), the value assigned to the new source $u$ is
 $\Phi_{\alpha,\vec T}(u)=q$.
  Moreover, $od_{\vect}(u)=s\le r-1$; each old source gains one incoming arc and therefore has underlying degree at most $r$; and all other degree bounds are unchanged.  Therefore $\vect\in\sett_{r,\alpha}(q)$.
\end{proof}

For any directed tree $\vect$,
let $\Delta(\vect)$ 
be the maximum degree of 
its underlying graph. 

\begin{lemma}\label{lem-3.2}
Let $r\in \N$ and 
$\alpha\ge 2$.
If there exists a multiset $\{q_1, q_2, \ldots, q_s\}$ of
elements of $\scrl_r(\alpha)$, where $1\le s\leq r$, such that 
$\sum\limits_{1\leq i\leq s}q_i^{-1}=\alpha-s$, then 
there exists $\vect \in 
\sett_{r, \alpha}(-1)$
with $s\le \Delta(\vect)\le r$.
\end{lemma}

\begin{proof}
 Assume that 
$\{q_1, q_2, \ldots, q_s\}$ is 
a multiset of $\scrl_r(\alpha)$, where $1\le s\leq r$, such that 
$\sum\limits_{1\leq i\leq s}q_i^{-1}=\alpha-s$.
By Lemma~\ref{lem-3.0}, 
there exists $\vect_i\in 
\sett_{r, \alpha}(q_i)$
for each $i\in \brk{s}$.
Consider $\vect_1, \vect_2,\ldots, \vect_s$ 
as pairwise vertex-disjoint copies.
Assume that 
 $\Phi_i$ is the mapping $\Phi_{\alpha, \vect_i}$ and $u_i$ is the unique source of $\vect_i$ for all $i\in \brk{s}$.
By definition, $\Phi_i(u_i)=q_i$ for each $i\in \brk{s}$.

Let $\vect$ be the directed tree obtained from $\vect_1, \vect_2, \ldots, \vect_s$ by adding a new vertex $u$ and adding a new directed edge $(u, u_i)$ for each $i\in \brk{s}$,
as shown in Figure~\ref{branch}.
Let $\Phi$ be the mapping defined by $\Phi(u)=-1$ and $\Phi(w)=\Phi_i(w)$ 
for all $w\in V(\vect_i)$, 
where $1\le i\le s$.

We now show that 
$\Phi$ satisfies condition (\ref{eq-7}) for every vertex $v$
in $\vect$.
For every $w\in V(\vect)\setminus \{u\}$,
$w\in V(\vect_i)$ for some $i\in \brk{s}$, and thus 
$\Phi(w)=\Phi_i(w)$ and 
$N^+_{\vect}(w)
=N^+_{\vect_i}(w)$.
It follows that for every $v\in V(\vect_i)$,
where $i\in \brk{s}$, 
as $\Phi_i$
satisfies condition (\ref{eq-7})
at $v$, 
$\Phi$ also 
satisfies condition (\ref{eq-7})
at $v$.
It remains to show that 
$\Phi$ 
satisfies condition (\ref{eq-7})
at vertex $u$.

By the given condition, we have 
$\sum\limits_{1\leq i\leq s}q_i^{-1}=\alpha-s$.
Thus,
\equ{lem-3.2-e1}
{
\Phi(u)=-1=\alpha-1-s-
\sum_{1\leq i\leq s}q_i^{-1}
=\alpha-1-s-
\sum_{1\leq i\leq s}\Phi(u_i)^{-1}.
}
Hence $\Phi$ is a mapping from $V(\vect)$ to $\{-1\}\cup \mathbb{R}^+$ satisfying condition (\ref{eq-7}).
It follows that $\vect\in \sett_{r,\alpha}(-1)$.
By the construction of $\vect$, $\Delta(\vect)\ge od_{\vect}(u)=s$.
Since $s\leq r$ and $\vect_i\in \sett_{r, \alpha}(q_i)$ for each $i\in \brk{s}$, 
we have $\Delta(\vect)\leq r$.
\end{proof}


For any $\vect\in \sett$,
let $\mu(\vect)=\mu(T)$ 
for its underlying tree $T$.
Next, we can show 
$\mu(\vect)=\alpha$ for each 
$\vect\in \sett_{r,\alpha}(-1)$.

\begin{lemma}\label{lem-3.3}
Let  $r\in \N$
and 
$\alpha\ge 2$.
For any $\vect\in \sett_{r, \alpha}(-1)$,
$\mu(\vect)=\alpha$.
\end{lemma}

\begin{proof}
If $\alpha=2$ and $\vect\in \sett_{r, \alpha}(-1)$, 
it can be verified by the definition 
of $\sett_{r, \alpha}(-1)$ 
and (\ref{eq-7}) 
that $\vect$ is the directed tree 
of order $2$.
Then $\mu(\vect)=2$, 
and the result holds.

Now assume that $\alpha>2$
and $\vect\in \sett_{r, \alpha}(-1)$.
By definition,
$\vect$ is a directed tree with 
a unique source, say $u$.
By Lemma \ref{lem3.1},
$\Phi_{\alpha, \vect}(v)\in \scrl_{r}(\alpha)$ for each $v\in V(\vect)\setminus \{u\}$.
Let $\vect_1, \vect_2, \ldots, \vect_s$ be the components of $\vect-u$,
where $1\leq s\leq r$ and 
$\vect_i\in \sett_{\alpha}$ for each $i\in \brk{s}$.
Let $u_i$ be the unique 
source of $\vect_i$ and 
$q_i=\Phi_{\alpha, \vect}(u_i)$ for each $i\in \brk{s}$.
By (\ref{eq-7}), we have 
\begin{equation}\label{eq-8}
-1=\Phi_{\alpha, \vect}(u)=\alpha-1-s-\sum_{i=1}^{s}q_i^{-1}.
\end{equation}
In order to prove that $\mu(\vect)=\alpha$, 
by Proposition \ref{prop-1}, 
it suffices to show that the mapping $\phi: V(\vect)\rightarrow \R$ defined below satisfies condition  (\ref{eq-3}):
\begin{enumerate}
	[itemsep=0mm, label=(\roman*)]
\item $\phi(u)=1$; and
\item for each directed edge 
$(w_1, w_2)$ in $\vect$, 
$\phi(w_2)=\frac{\phi(w_1)}{\Phi_{\alpha, \vect}(w_2)}$.
\end{enumerate} 

The map $\varphi$ is well-defined.  Indeed, every
$v\in V(\vec T)\setminus\{u\}$ is reached from $u$ by a unique directed path
$u=v_0,v_1,\ldots,v_m=v$, and hence
$
\varphi(v)=\prod\limits_{j=1}^{m}\Phi_{\alpha,\vec T}(v_j)^{-1}.
$
Each $v_j$ is a non-source vertex, so
$\Phi_{\alpha,\vec T}(v_j)>0$ by the definition of
$\vec{\mathcal T}_{r,\alpha}(-1)$.  Consequently $\varphi(v)>0$ for every
$v$, and thus $\varphi:V(\vec T)\to\mathbb \R$ as claimed.

We can first show that condition (\ref{eq-3}) is satisfied for the unique source and all sinks of $\vect$.
As $u$ is the unique source 
of $\vect$ and 
$N_{\vect}^+(u)=\{u_i: 1\leq i\leq s\}$,
we have 
\begin{equation}\label{eq-9}
\sum_{z\in N_{\vect}(u)}\phi(z)
=\sum_{i=1}^{s}\phi(u_i)
=\sum_{i=1}^{s}\frac{\phi(u)}
{\Phi_{\alpha,{\vect}}(u_i)}
=\phi(u)\sum_{i=1}^{s}q_i^{-1}
=(\alpha-s)\phi(u),
\end{equation}
where the last step follows from (\ref{eq-8}).

For any sink $w$ of ${\vect}$, $\Phi_{\alpha,{\vect}}(w)=\alpha-1$ by (\ref{eq-7}).
If $w'$ is the unique vertex in $N^-_{\vect}(w)$, 
where $N^-_{\vect}(w)$ is the set of in-neighbors of $w$ in ${\vect}$,
then $\phi(w')=\Phi_{\alpha,{\vect}}(w)\phi(w)$ by the definition of $\phi$,
and 
\begin{equation}\label{eq-10}
\sum_{z\in N_{\vect}(w)}\phi(z)=\phi(w')
=\phi(w)\Phi_{\alpha,{\vect}}(w)=(\alpha-1)\phi(w).
\end{equation}

It remains to consider any vertex $w$ in ${\vect}$ which is neither the source nor a sink of ${\vect}$.
Let $N_{\vect}^{-}(w)=\{w'\}$ and 
$N_{\vect}^{+}(w)=\{w_1, w_2, \ldots, w_t\}.$ 
Then, by the definition of $\phi$,
\eqn{eq-11}
{
\phi(w')+\sum_{i=1}^t\phi(w_i)
&=&\Phi_{\alpha,{\vect}}(w)\phi(w)
+\sum_{i=1}^{t}
\frac{\phi(w)}{\Phi_{\alpha,{\vect}}(w_i)}
\nonumber \\
&=&\phi(w)\Bigg(\Phi_{\alpha,{\vect}}(w)
+\sum_{i=1}^{t}\Phi_{\alpha,{\vect}}(w_i)^{-1}\Bigg)
\nonumber \\
&=&(\alpha-1-t)\phi(w),
}
where the last step follows from (\ref{eq-7}).
Thus, condition (\ref{eq-3}) holds for all vertices in ${\vect}$, implying that $\mu({\vect})=\alpha$ by Proposition \ref{prop-1}.
\end{proof}

\subsection{Proofs of 
	Theorems~\ref{se4-th}
	and~\ref{thm-1}}

Now we apply 
Corollary~\ref{se2-co1}, 
Lemmas~\ref{lem-3.2} and 
\ref{lem-3.3}  to prove Theorem~\ref{se4-th}.

\vspace{2 mm}

\noindent {\it Proof of Theorem~\ref{se4-th}}.
Assume that $T$ is a tree with 
$\Delta(T)=r$ and $\mu(T)=\alpha\ge 2$.
By Corollary~\ref{se2-co1},
there exists a multiset 
$\{q_1,q_2,\ldots,q_r\}$ of elements of $\scrl_r(\alpha)$ 
such that $\sum_{i=1}^rq_i^{-1}=\alpha-r$.

Now assume that there exists a multiset 
$\{q_1,q_2,\ldots,q_r\}$ of elements of $\scrl_r(\alpha)$ 
such that $\sum_{i=1}^rq_i^{-1}=\alpha-r$.
By Lemma~\ref{lem-3.2},
there exists $\vect\in \sett_{r,\alpha}(-1)$
with $\Delta(\vect)=r$.
Then, by Lemma~\ref{lem-3.3},
$\mu(\vect)=\alpha$.

Hence Theorem~\ref{se4-th} holds.
\proofend 

We now prove the following proposition, from which Theorem \ref{thm-1} follows directly.

\begin{proposition}\label{prop-3.4}
For any $\alpha\ge 2$ and
$r\in \N$, 
the following statements are pairwise equivalent:
\begin{enumerate}[itemsep=0mm, label=(\roman*)]
\item there exists a tree $T$ with $\Delta(T)\leq r$ and $\mu(T)=\alpha$;
\item $(\alpha-1)^{-1}\in \scrl_r(\alpha)$;
\item there exists $a,b\in \scrl_r(\alpha)$ such that $ab=1$;
and 

\item 
there exists a multiset 
 $\{q_1, q_2, \ldots, q_s\}$
 of elements of $\scrl_r(\alpha)$, where $1\le s\leq r$, such that $\sum\limits_{i=1}^{s}q_i^{-1}=\alpha-s$.
\end{enumerate}
\end{proposition}

\begin{proof}
We prove the cycle of 
implications.

(i)$\Rightarrow$(ii).
It follows from Corollary~\ref{se2-co1}.


(ii)$\Rightarrow$(iii).
Assume that $(\alpha-1)^{-1}\in \scrl_r(\alpha)$. 
Since $\alpha-1\in \scrl_r(\alpha)$,
(iii) holds. 

(iii)$\Rightarrow$(iv).
Assume that $ab=1$ for 
$a,b\in \scrl_r(\alpha)$.
If $a=b=\alpha-1$, then
$ab=1$ implies that $\alpha=2$.
It follows that $1\in \scrl_r(\alpha)$.
Then $q_1=1\in \scrl_r(\alpha)$
and $q_1^{-1}=\alpha-1$.

Now consider the case $a\ne \alpha-1$ or 
$b\ne \alpha-1$.
We may assume that $a\ne \alpha-1$.
As $a\in \scrl_r(\alpha)\setminus \{\alpha-1\}$,
there is a minimum $n\in \N$ such
that $a\in \scrl^{(n)}_r(\alpha)$.
By definition, there exists
$q_1,q_2,\ldots,q_t\in 
\scrl^{(n-1)}_r(\alpha)\subseteq \scrl_r(\alpha)$, where $1\le t\le r-1$,
such that 
\equ{prop-3.4-e1}
{
a=\alpha-1-t-\sum_{i=1}^t q_i^{-1}.
}
Let $s=t+1$ and $q_{s}=b$. Then
\equ{prop-3.4-e2}
{
\sum_{i=1}^s q_i^{-1}
=\alpha-t-1-a+\frac 1{b}
=\alpha-t-1=\alpha-s.
}
Hence (iv) follows from (iii).

(iv)$\Rightarrow$(i).
It follows directly from 
Lemmas \ref{lem-3.2} and \ref{lem-3.3}.

Hence the result holds.
\end{proof}

By 
Proposition~\ref{prop-3.4},
Theorem~\ref{thm-1} follows
directly.

\section{Proof of Theorem~\ref{thm-2}
	\label{sec4}
} 

Let $k\in \N$ with $k\ge 2$.
By (\ref{eq-1}), if $T$ is a tree 
with $\mu(T)=k^2$, then 
$(k-1)^2+2\le \Delta(T)\le k^2-1$.
Let $r_0=(k-1)^2+2$.
In this section, we first
establish several preliminary 
conclusions on the elements of  
$\scrl_{r_0}(k^2)$, and then 
apply Theorem~\ref{se4-th}
to complete the proof of 
Theorem~\ref{thm-2}.

For  any $x_1,x_2,\ldots,x_{s}\in \R$, 
let 
\equ{gamma-f}
{
\Gamma(x_1,x_2,\ldots,x_s)
=k^2-1-s-\sum_{i=1}^sx_i^{-1}.
}
Let $x^{[s]}$ 
denote $s$ copies of $x$. 
Then, 
$	\Gamma(x^{[s]})
	=k^2-1-s-s\, x^{-1}$.

\begin{lemma}
	\label{lem-gamma}
	For $1\le s\le r_0-1$
	and $x\in \scrl_{r_0}(k^2)$,
	if
	$\Gamma(x^{[s]})>0$, then 
	$\Gamma(x^{[s]})\in 
	\scrl_{r_0}(k^2)$.
\end{lemma}

\proof 
 This follows directly from the defining recursion for
$\scrl_{r_0}(k^2)$, applied to $s$ copies of $x$.
\proofend

For any integer $i\ge 0$, define 
\equ{fuct-h}
{
h(i)
=\frac {((i + 1)k + 1)(k - 1)}
{ik + 1}
}
and 
\equ{}
{
	p(i)=\frac{\left(k-1\right) \left(k^{3}-2 k^{2}-\left(i-1\right) k-1\right)}{k^{3}-2 k^{2}-\left(i-2\right) k-1}.
}

\lemm{lem4-1}
{
Let $k\in \N$ with $k\ge 2$.
Then 
\begin{enumerate}
	[itemsep=0mm, label=(\roman*)]
\item 
$h(i)\in 
\scrl_{r_0}(k^2)$ for all integers $i\ge 0$; 

\item 
$p(i)\in 
\scrl_{r_0}(k^2)$ for all integers $i$ with $0\le i\le (k-1)^2-1$; and

\item 
$\frac{(k-1)^2}{2k-1}\in  \scrl_{r_0}(k^2)$.

\end{enumerate} 
}

\proof (i) Obviously, $h(i)>0$
for all $i\ge 0$.
Note that $h(0)=k^2-1\in \scrl_{r_0}(k^2)$.
Assume that $h(i)\in \scrl_{r_0}(k^2)$, where $i\ge 0$.
Direct computation shows that 
$$
	h(i+1)=\frac {((i + 2)k + 1)(k - 1)}
	{(i+1)k + 1}
	=
	k^2-1- (k-1)^2-
	\frac{(k-1)^2}{h(i)}
	=\Gamma(h(i)^{[r_0-2]})
	\in \scrl_{r_0}(k^2),
$$
where the last step 
follows from Lemma~\ref{lem-gamma}.
By definition,  the result holds.

(ii) For any integer $i$ with 
$-1\leq i\leq(k-1)^2-1$,  
it can be verified  
\equ{le42-2-e1}
{
	k^3-2k^2-(i-1)k-1\geq k-1>0.
}
Thus, by (\ref{le42-2-e1}), 
for $0\le i\le (k-1)^2-1$,
both numerator and denominator of $p(i)$ are positive, implying that $p(i)>0$.

By (i), $h(k-2)
=\frac{k^2-k+1}{k-1}\in 
\scrl_{r_0}(k^2)$.
Direct 
computation shows that 
$$
	p(0)=\frac{\left(k-1\right) \left(k^{3}-2 k^{2}+k-1\right)}{k^{3}-2 k^{2}+2 k-1}
	=k^2-1-(r_0-1)-\frac{r_0-1}
	{\frac{k^{2}-k+1}{k-1}}
	=\Gamma(h(k-2)^{[r_0-1]})
	\in \scrl_{r_0}(k^2).
$$
Assume inductively that $p(i)\in\scrl_{r_0}(k^2)$, 
where 
$0\le i\le (k-1)^2-2$.
Then, 
$$
	p(i+1)=\frac{\left(k-1\right) \left(k^{3}-2k^2-ik-1\right)}{k^{3}-2 k^{2}-(i-1)k-1}
	=k^2-1-(r_0-2)- \frac{r_0-2}{p(i)}
	=\Gamma(p(i)^{[r_0-2]})
	\in \scrl_{r_0}(k^2).
$$
Thus (ii) holds.

(iii) By (ii), 
$p((k-1)^2-1)\in \scrl_{r_0}(k^2)$.
Direct computation shows that 
\equ{p6-01-e4}
{
p((k-1)^2-1)=
\frac{\left(k-1\right)\cdot  \left(k^{3}-2\cdot k^{2}-\left((k-1)^2-2\right)\cdot k-1\right)}{k^{3}-2\cdot k^{2}-\left((k-1)^2-3\right)\cdot k-1}
=\frac{\left(k-1\right)^{2}}{2 k-1}.
}
Thus,  
$\frac{\left(k-1\right)^{2}}{2 k-1}\in \scrl_{r_0}(k^2)$ by (ii), and (iii) holds.
\proofend 

We now 
prove Theorem~\ref{thm-2}.

\vspace{2 mm}

\noindent 
{\it Proof of Theorem~\ref{thm-2}}.
By (\ref{eq-1}), 
if $T$ is a tree with $\mu(T)=k^2$ and $\Delta(T)=r$,
then $(k-1)^2+2\le r\le k^2-1$.

Now we assume that $k\ge 2$
and $r$ is any integer 
with 
$(k-1)^2+2\le r\le k^2-1$.
In order to show the existence 
of a tree $T$ with $\mu(T)=k^2$ and $\Delta(T)=r$, 
by Theorem~\ref{se4-th}, 
it suffices to 
prove the following claim.

\noindent {\bf Claim *}:
There exists a  multiset $\{q_1,q_2,\ldots,q_r\}$ 
	of elements of $\scrl_r(k^2)$ 
	such that 
\equ{thm-2-e1}
{
\sum_{i=1}^r q_i^{-1}=k^2-r.
}
Set $t:=k^2-r$. 
So $r=k^2-t$.
As
$(k-1)^2+2\le r\le k^2-1$, 
we have 
$1\le t\le 2k-3$. 

\incase $1\le t\le k-1$.

Note that $r_0=(k-1)^2+2$.
Thus, $r_0\le r$, implying that  
$\scrl_{r_0}(k^2)\subseteq 
\scrl_r(k^2)$.
By Lemma~\ref{lem4-1}, 
$h(i)\in \scrl_{r_0}(k^2)\subseteq 
\scrl_r(k^2)$ for all 
$i\ge 0$.
Define
\begin{equation}
	x\coloneqq (k+1)(k-t-1),
	\qquad
	y\coloneqq tk+1,
	\qquad
	j\coloneqq t-1.
	\label{thm-2-e2}
\end{equation}

The assumptions of this case imply
$x\geq 0,\, y>0$ and 
$j\geq 0$.
We 
take \(x\) copies of \(h(0)\) and \(y\) copies of \(h(j)\),
where $x=0$ is allowed.
The number of elements
in the multiset is
\begin{align}
	x+y=(k+1)(k-t-1)+tk+1
	=k^2-t=r.
	\label{thm-2-e3}
\end{align}
Direct computation shows that 
\equ{thm-2-e4}
{
	\frac{x}{h(0)}+\frac{y}{h(j)}
	=\frac{(k+1)(k-t-1)}{(k+1)(k-1)}+\frac{(tk+1)((t-1)k+1)}{(tk+1)(k-1)}
	=\frac{t(k-1)}{k-1}
	=k^2-r.
}
Hence Claim *
holds in Case~1.

\incase \(k\leq t\leq 2k-3\).

By Lemma~\ref{lem4-1}, 
$a:=\frac{(k-1)^2}{2k-1}\in 
\scrl_{r_0}(k^2)\subseteq 
\scrl_r(k^2)$.
Define
\begin{equation}
	u\coloneqq (k-1)(t-k+1),
	\qquad
	v\coloneqq k(2k-t-2)+1,
	\qquad
	\ell\coloneqq 2k-t-3.
	\label{thm-2-e5}
\end{equation}
The assumptions of this case imply
$u\geq 0,\, v>0$ and $
	\ell\geq 0$.
	
Take \(u\) copies of \(a\) and \(v\) copies of \(h(\ell)\).
Note that  
\begin{align}
	u+v
	=(k-1)(t-k+1)+k(2k-t-2)+1
	=k^2-t=r.
	\label{thm-2-e6}
\end{align}
Direct computation 
shows that 
\eqn{thm-2-e7}
{
	\frac{u}{a}+\frac{v}{h(\ell)}
	&=&\frac{(k-1)(t-k+1)(2k-1)}{(k-1)^2}
	+\frac{k(2k-t-2)+1}{\frac{\left(k-1\right) \left(2 k^{2}-k t-2 k+1\right)}{2 k^{2}-k t-3 k+1}}
\nonumber \\
&=&\frac{-2k^2+(2t+3)k-t-1}
{k-1}
+\frac{2 k^{2}-k t-3 k+1}{k-1}
	=t=k^2-r.
}
Hence Claim *
holds in Case~2.

Therefore, by Claim *
and Theorem~\ref{se4-th}, 
 there exists
a tree $T$ with $\mu(T)=k^2$ and $\Delta(T)=r$.
Hence Theorem~\ref{thm-2} holds.
\proofend

\section{Further Research
\label{sec5}
}

By~(\ref{eq-1}), every tree $T$ with $\mu(T)=\alpha\ge 4$ 
and $\Delta(T)=r$
satisfies
\equ{eq7-1}
{
	\alpha+2-2\sqrt{\alpha}
	<r\le \alpha-1.
}
This naturally leads to the following problem.

\prom{prob7-1}
{
	Let $\alpha$ and $r$ be positive integers with $\alpha\ge 4$. 
	If
	\[
	\alpha+2-2\sqrt{\alpha}<r\le \alpha-1,
	\]
	does there always exist a tree $T$ such that
	$\mu(T)=\alpha$ and $\Delta(T)=r$?
}


In the present article, we resolve Problem~\ref{prob7-1} only for
$\alpha=k^2$, where $k\ge 2$ is an integer; see
Theorem~\ref{thm-2}.
We are not aware of a solution beyond
the square case.

\vspace{3mm} 

\noindent {\bf Conflict of interest}: The authors declare that they have no known competing financial interests or personal
relationships that could have appeared to influence the work reported in this paper.

\noindent {\bf Data availability statement}:  Not applicable.

\section*{Acknowledgements}
This research is supported by 
the NSFC (Nos. 12101347 and 12371340)
and the Natural Science Foundation 
of Shandong Province (No. ZR2021QA085).

\def \LAA {{\it Linear Algebra Appl.} }
\def \JCTB {{\it Combin. Theory, Ser. B} }
\def \JGT {{\it J. Graph Theory} }
\def \GC {{\it Graphs Comb.} }
\def \DM {{\it Discrete Math.} }
\def \JMC {{\it J. Math. Chem.} }
\def \AJC 
{{\it Australas. J. Combin.} }
\def \EJC 
{{\it Electron. J. Comb.} }
\def \EUJC 
{{\it Eur. J. Comb.} }
\def \SIAM
{{\it SIAM J. Discrete Math.} }

\end{document}